\documentclass[11pt,a4paper,reqno]{amsart}

\usepackage{amscd,amsmath}
\usepackage{amssymb,amsmath,latexsym,color,enumerate,tikz,dsfont}
\usepackage[all]{xy}       
\usepackage{mathabx}
\usepackage{mathrsfs}
\usepackage{setspace}
\usepackage[makeroom]{cancel}

\usepackage{floatrow}

\def\2{\text{\bf 2}}

\newcommand{\twoheaddownarrow}{{\rlap{\rlap{$\ $}\raise .25ex\hbox{$\downarrow$}}\raise-.25ex\hbox{$\downarrow$}}}
\newcommand{\twoheaduparrow}{{\rlap{\rlap{$\ $}\raise .25ex\hbox{$\uparrow$}}\raise-.25ex\hbox{$\uparrow$}}}

 \newcommand{\tbigcap}{\mathop{\textstyle \bigcap }}

 \newcommand{\tbigvee}{\mathop{\textstyle \bigvee }}
 \newcommand{\tbigwedge
}{\mathop{\textstyle \bigwedge }}

\newcommand{\dirvee}{\mathop{\setminus\kern-.85ex\nnearrow}}

\DeclareMathAccent{\widetriangle}{\mathord}{largesymbols}{"E6}

\newcommand{\po}{\ar@{}[dr]|{\text{\pigpenfont R}}}
\newcommand{\pb}{\ar@{}[dr]|{\text{\pigpenfont J}}}

\newtheorem{theorem}{Theorem}[section]
\newtheorem{proposition}[theorem]{Proposition}
\newtheorem{lemma}[theorem]{Lemma}
\newtheorem{corollary}[theorem]{Corollary}

\theoremstyle{definition}
\newtheorem{definition}[theorem]{Definition}
\newtheorem{example}[theorem]{Example}

\theoremstyle{remark}

\newtheorem{remark}[theorem]{Remark}

\begin{document}

\title{On Isbell's Density Theorem\\
for bitopological pointfree spaces II}



\thanks{The second named author acknowledges support from the Basque Government (grant IT1483-22).}

\author[M.\thinspace A. Moshier]{M. Andrew Moshier}
\address[M. Andrew Moshier]{CECAT\\ Department of Mathematics \& Computer Science\\ Chapman University\\ Orange CA 92866\\ U.S.A.}
\email{moshier@chapman.edu}

\author[I. Mozo~Carollo]{Imanol Mozo~Carollo}
\address[Imanol Mozo~Carollo]{Department of Applied Economics\\ University of the Basque Country UPV/EHU\\ 20018 Donostia\\ Spain}
\email{imanol.mozo@ehu.eus}

\author[J. Walters-Wayland]{Joanne Walters-Wayland}
\address[Joanne Walters-Wayland]{CECAT\\ Department of Mathematics \& Computer Science\\ Chapman University\\ Orange CA 92866\\ U.S.A.}
\email{joanne@waylands.com}

%

\date{\today}

\setstretch{1.3}


\begin{abstract}
With the aim of studying subspaces in pointfree bitopology, we characterize extremal epimorphism in biframes and show that a smallest dense one always exists, providing an analogue of Isbell's Density Theorem. Further we study the functoriality of assigning to each biframe its lattice of subbilocales and its smallest dense subbilocale.
\end{abstract}
\maketitle

\section{Introduction}

This paper is a sequel to a previous work \cite{MMW20} wherein we addressed the study of subspaces in pointfree bitopology in terms of $d$-frames \cite{JM06}. We now turn our attention towards a more classical approach to bitopological pointfree spaces, namely, \emph{biframes} \cite{BBH83}. 

The reader will be familiar with the basics. 
The study of pointfree topology can be understood as the study of the category $\mathsf{Frm}$ of frames---complete lattices in which finite meets distribute over arbitrary joins---and frame homomorphisms---maps between frames preserving finite meets and arbitrary joins---with the caveat that the relation between frames and classical topology is contravariant.

The most studied pointfree approach to bitopological spaces was presented by Banaschewski, Brummer and Hardie in \cite{BBH83}. There pointfree bitopological spaces are \emph{biframes}. Biframes consist of three frames $L=(L_0,L_1,L_2)$  where $L_1$ and $L_2$ are subframes of $L_0$ and their union $L_1\cup L_2$ forms a subbasis of $L_0$. They form the category $\mathsf{BiFrm}$ of biframes together with biframe homomorphisms $f\colon L\to M$, which are given by frame homomorphisms $f_0\colon L_0\to M_0$ such that they restrict to frame homomorphisms $f_0|_{L_i}=f_i\colon L_i \to M_i$ for $i=1,2$. For simplicity, by a slight abuse of notation, we may use $f$ to denote the underlying frame homomorphism $f_0$ if there is no risk of confusion.

In \cite{MMW20} we provided a bitopological version of Isbell's Density Theorem \cite{I72} for $d$-frames. Isbell's celebrated result states that every pointfree space $L$ has a smallest dense ``subspace''.  This has to do with what we mean by ``subspace'' in this context.  In this setting, a pointfree topological embedding translates to an extremal epimorphism in $\mathsf{Frm}$ or, equivalently, to extremal monomorphism in its dual category $\mathsf{Loc}$, the category of \emph{locales}. ``Subspaces'' in this context means \emph{sublocales}, i.e., subobjects in $\mathsf{Loc}$ (extremal subobjects, strictly speaking). Recall that a frame homomorphism $f\colon L\to M$ is said to be \emph{dense} if $f(a)=0$ implies $a=0$. The collection of all regular elements $a=a^{\ast\ast}$ of a frame $L$ forms a complete Boolean algebra, denoted by $\mathfrak{B}(L)$, the \emph{Booleanization} of $L$. The smallest dense sublocale of $L$ is represented by the extremal epimorphism $\beta_L\colon L\to\mathfrak{B}(L)$ that maps each element $a$ to its double pseudocomplement $a^{\ast\ast}$.

Interestingly, the concepts of least dense sublocale and of Booleanization coincide for frames, but this is not the case in pointfree bispaces in general, as  Suarez shows \cite{S22} in the context of  d-frames. In this paper we focus  on the notion of  least dense  ``subspace''.  We provide a biframe version of Isbell's Density theorem and show that each biframe has a smallest dense extremal epimorphism and explore how to make our construction functorial by restricting the  hom-sets (Section \ref{IsbellsDensity}). With this aim in mind, we first characterize extremal epimorphisms and establish an extremal epi-mono factorization system in $\mathsf{BiFrm}$ (Section \ref{extepischar}). We define \emph{subbilocales} of a biframe $L$ as the equivalence classes of extremal epimorphisms with $L$ as domain. Then we show that the poset $\mathcal{S}(L)$ of subbilocales of $L$  is a complete lattice (Section \ref{sectionlatq}), although it is not distributive in general, as it is shown in a counterexample in Section \ref{sectioninducedsubs}. In Section \ref{sectionfunctorQ} we show that the assignment $L\mapsto \mathcal{S}(L)$ is functorial and extends the functor that maps frames to their coframes of sublocales.

\section{Background and notation}

\subsection{Posets}
For any subset $A$ of a partially ordered set $(P,\le)$, we will denote by $
 {\tbigvee^P} A$ (resp. $
 {\tbigwedge^P} A$) the supremum (resp. infimum) of $A$ in $P$ in case it exists (we shall omit the superscript if it is clear from the context).

\subsection{Categories}
For a coproduct $A\oplus B$ we will denote the coproduct injections by $\iota_A\colon A\to A\oplus B$ and $\iota_B\colon B\to A\oplus B$. For any morphisms $f\colon A\to C$ and $g\colon B\to D$, we will denote by $f\oplus g$ the unique frame homomorphism $A\oplus B\to C\oplus D$ that makes
 \[
 \xymatrix@R+1pc@C+2pc{A\ar[d]_{f} \ar[r]^-{\iota_{A}}&A\oplus B\ar[d]_{f\oplus g} &\ar[l]_-{\iota_{B}} B \ar[d]^{g}\\
C \ar[r]_-{\iota_{C}}& C\oplus D &\ar[l]^-{\iota_{D}} D}
\]
commute.

Given an object $B$, we say that a family of morphisms $\{f_i\colon A_i\to B\}_{i\in I}$ is \emph{jointly epimorphic} (or \emph{jointly epic} for short) if given any two morphisms $g, h\colon B\to C$ such that $g\cdot f_i=h\cdot f_i$ for all $i\in I$, it follows that $g=h$.

\subsection{Frames}

For general notions and results concerning frames we refer to Johnstone \cite{PJ82} or the more recent
Picado-Pultr \cite{PP12}. The latter is particularly useful for details about sublocales. Below, we provide a brief survey of the background required for this paper.

Since in any category with pushouts and pullbacks, extremal epimorphisms are stable under pushouts, and extremal epimorphisms in $\mathrm{Frm}$ are precisely onto frame homomorphism, it is the case that pushouts in $\mathsf{Frm}$ preserve onto frame homomorphisms.

A \emph{frame congruence}
 in a frame $L$ is an equivalence relation $\mathcal{R}$ respecting all joins and finite meets. Given a frame congruence $\mathcal{R}$ we can define the \emph{quotient frame}\index{quotient!frame}\index{frame!quotient $\sim$} $L/\mathcal{R}$ just as in algebraic fashion: the elements are the $\mathcal{R}$-classes
\[
\langle a\rangle=\{b\in L\mid (b,a)\in\mathcal{R}\}.
\]
 There is an extremal epimorphism (onto frame homomorphism) $p\colon L\to L/\mathcal{R}$ given by $a\mapsto \langle a\rangle$ for each $a\in L$. Further, if $h\colon L\to M$ is a frame homomorphism such that  $h(a)=h(b)$ for all $a,b\in L$ satisfying $(a,b)\in \mathcal{R}$, then there is a unique frame homomorphism $\tilde{h}\colon L/\mathcal{R}\to M$ such that $\tilde{h}\cdot p=h$.

We will denote by $\mathcal{C}(L)$ the set of all frame congruences of $L$. Ordered by inclusion, $\mathcal{C}(L)$ is a frame and infimum is given by intersection. This fact allows us to define the \emph{frame congruence generated}\index{congruence!generated by a relation} by $R$ for any $R\subseteq L\times L$,
\[
[R]=\tbigwedge\{\mathcal{R}\in\mathcal{C}(L)\mid R\subseteq \mathcal{R}\}=\tbigcap\{\mathcal{R}\in\mathcal{C}(L)\mid R\subseteq\mathcal{R}\}
\]
the least frame congruence containing $R$. Further given a frame homomorphism $h\colon L\to M$ such that $h(a)=h(b)$ for all $(a,b)\in R$ one has that $h(c)=h(d)$ for all $(c,d)\in [R]$.

\subsection{Biframes}

Given a biframe $L=(L_0,L_1,L_2)$, we will call $L_0$ the \emph{ambient} frame  and $L_1$ and $L_2$  the \emph{component} frames of $L$. Other  authors (see \cite{S92}, for example) call them \emph{total}, \emph{first}  and \emph{second parts}, respectively,  but  we will stick to our terminology as we find it useful to have a single  term, component frame, to refer to either of the frames $L_1$ and $L_2$. If we identify each component frame $L_i$ with its isomorphic image under the coproduct injection $\iota_i\colon L_i\to L_1\oplus L_2$, then $(L_1\oplus L_2, L_1, L_2)$ is a biframe. Furthermore, for any biframe $L$, it is immediate to see that the ambient frame $L_0$ has to be a quotient of the coproduct $L_1\oplus L_2$ of its ambient frame. We will denote by $q_L$ the quotient map $L_1\oplus L_2\to L_0$ such that $e_i=\iota_{L_i}\cdot q_L$ for $i=1,2$.

Note that $L_1$ and $L_2$ together generating $L_0$ implies that the inclusion maps $e_{L_1}$ and $e_{L_2}$ are jointly epic. In particular, the coproduct injections $\iota_{L_i}\colon L_i\to L_1\oplus L_2$ are jointly epic.

A biframe homomorphism $h$ is monic in $\mathsf{BiFrm}$ iff $h_1$ and $h_2$ are both one-one and  $h$ being one-one implies that $h$ is a monic, but not conversely  \cite{S92}.

\section{Extremal epimorphisms in $\mathsf{BiFrm}$}\label{extepischar}

In this section, we characterize extremal epimorphisms and establish an extremal epi-mono factorization system for biframes. To this end, let $f\colon  L\to M$ be a biframe homomorphism and let $e_i$ denote the inclusion maps of $f[L_i]$ into $M_0$. There is a unique frame homomorphism $g\colon f[L_1]\oplus f[L_2]\to M_0$ making the following diagram commute:
\[
\xymatrix@R+1.pc@C+1.pc{
f[L_1]\ar[r]^-{\iota_{f[L_1]}}\ar[dr]_{e_1}&f[L_1]\oplus f[L_2]\ar@{.>}[d]^g&f[L_2]\ar[l]_-{\iota_{f[L_2]}}\ar[dl]^{e_2}\\
&M_0&
}
\]
For $i=1,2$, we have
\[
\begin{aligned}
g\cdot f_1\oplus f_2\cdot \iota_{L_i}&=g\cdot\iota_{f[L_i]}\cdot f_i\\
&=e_i\cdot f_i\\
&=f\cdot q_L\cdot \iota_{L_i}.
\end{aligned}
\]
Therefore $g\cdot f_1\oplus f_2=f\cdot q_L$, as $\iota_{L_1}$ and $\iota_{L_2}$ are jointly epic. 
Hence, there is a unique $e\colon f(L)\to M_0$ such that 
\[
\xymatrix{
L_1\oplus L_2\ar[r]^-{q_L}\ar[d]_{f_1\oplus f_2} & L_0\ar[d]^{\overline{f}}\ar@/^1.5pc/[ddr]^{f}&\\
f[L_1]\oplus f[L_2]\ar@/_1.5pc/[drr]_g\ar[r]_-{\overline{q_L}}&f(L)\ar@{.>}[dr]^e&\\
&&M_0}
\]
commutes, 
where the inner square is a pushout diagram.

Consequently $e\cdot \overline{q_L}\cdot\iota_{f[L_i]}=g\cdot \iota_{f[L_i]}=e_i$. Since $e_i$ is one-one, so is $\overline{q_L}\cdot\iota_{f[L_1]}$. Therefore $f[L_i]$ embeds into $f(L)$. Accordingly, we can identify $f[L_i]$ with its isomorphic image in $f(L)$. Furthermore, as pushouts preserve onto frame homomorphisms, $\overline{q_L}$ is onto and consequently $f[L_1]\cup f[L_2]$ forms a subbasis of $f(L)$. Hence $(f(L),f[L_1],f[L_2])$ is a biframe.  The following commutative diagram depicts our construction:
\[
\xymatrix{
L_1\ar[rr]^{\iota_{L_1}}\ar[dd]_{f_1}&&L_1\oplus L_2\ar[dr]^{q_L}\ar[dd]_{f_1\oplus f_2}&&L_2\ar[ll]_{\iota_{L_2}}\ar[dd]^{f_2}\\
&&&L_0\ar[dd]\ar[dd]_(0.25){\overline{f}}\ar@/^2.3pc/[ddd]^f &\\
f[L_1]\ar[rr]^-{\iota_{f[L_1]}}\ar@/_1.5pc/[ddrrr]_{e_1}&&f[L_1]\oplus f[L_2]\ar[dr]^{\overline{q_L}}\ar@/_1pc/[ddr]_g&&f[L_2]\ar[ll]|!{[d];[l]}\hole_(0.6){\iota_{f[L_2]}}\ar@/^1pc/[ddl]^{e_2}\\
&&&f(L)\ar[d]^(0.4){e}&\\
&&&M_0&
}
\]
As
\[
\overline{f}\cdot q_L\cdot\iota_{L_i}=f_i\cdot\iota_{f[L_i]}\cdot \overline{q_L}
\]
we have that $\overline{f}$ restricts to a frame homomorphism
\[
 \overline{f}|_{L_i}\colon L_i\to f[L_i].
 \]
Similarly, as
\[
e\cdot \overline{q_L}\cdot\iota_{f[L_i]}=e_i,
\]
the restriction
\[
e|_{f[L_i]}\colon f[L_i]\to M_i
\]
is a frame homomorphism. In summary, we have the following result.

\begin{proposition}\label{factor}
Using the notation above, any biframe homorphism $f\colon L\to M$ factors as follows:
\[
\xymatrix{
L\ar[rr]^{f}\ar[dr]_{\overline{f}}&&M\\
&(f(L),f[L_1],f[L_2])\ar[ur]_e&
}
\]
\end{proposition}

For a more concrete description of the factorization above, we have the following result.

\begin{proposition}
Again, using the notation above,
 let $\mathcal{R}_f$ be the least frame congruence   of $f[L_1]\oplus f[L_2]$ that contains 
\[
\left(\tbigvee_{i\in I}f(a_i)\oplus f(b_i),\tbigvee_{j\in J}f(c_j)\oplus f(d_j)\right)
\]
whenever
\[
\tbigvee^{L_0}_{i\in I}a_i\wedge b_i=\tbigvee^{L_0}_{j\in J}c_j\wedge d_j
\]
with $a_i, c_i\in L_1$ and $b_j, d_j\in L_2$ for for all $i\in I$ and all $j\in J$. Then one has that

 \[f(L)\simeq f[L_1]\oplus f[L_2]/\mathcal{R}_f.\]
\end{proposition}

\begin{proof}
Let $q$ be the quotient map
\[
q\colon f[L_1]\oplus f[L_2]\to f[L_1]\oplus f[L_2]/\mathcal{R}_f
\] 
determined by $\mathcal{R}_F$. We can a define a frame homomorphism 
\[
g\colon L_0\to f[L_1]\oplus f[L_2]/\mathcal{R}_f
\] 
by mapping  $\tbigvee_{i\in I}a_i\wedge b_i$ to the equivalence class of $\tbigvee_{i\in I}f(a_i)\oplus f(b_i)$. Note that $g$ is well-defined by the very definition of $\mathcal{R}_f$ and it is immediate to check that $g$ preserves joins and finite meets. For any $a\in L_1$, one has that
\[
g(q_L(a\oplus 1))=g(a\wedge 1)=q(f(a)\oplus 1)=q((f_1\oplus f_2)(a\oplus 1)).
\]
Analogously, for any $b\in L_2$,
\[
g(q_L(1\oplus b))=q( (f_1\oplus f_2)(1\oplus b)).
\]
We conclude that $g\cdot q_L=q\cdot f_1\oplus f_2$

Now we will check that $f[L_1]\oplus f[L_2]/\mathcal{R}_f$  satisfies the universal property with $q$ and $g$. If $m\cdot q_L=n\cdot f_1\oplus f_2$ for some frame homomorphisms $m\colon L_0\to N$ and $n\colon f[L_1]\oplus f[L_2]\to N$, then $n$ has to turn the relation $\mathcal{R}_f$ into the identity in $M$. Hence there exists
\[
p\colon f[L_1]\oplus f[L_2]/\mathcal{R}_f\to M
\]
 such that $p\cdot q=n$. Further, $p\cdot g=m$ follows from the fact that $q_L$ is epic and
\[
p\cdot g\cdot q_L=p\cdot q\cdot f_1\oplus f_2=n\cdot f_1\oplus f_2=m\cdot q_L.\qedhere
\]
\end{proof}

\begin{proposition}\label{extremalchar}
Extremal epimorphisms in $\mathsf{BiFrm}$ are precisely the biframe homomorphisms $f\colon L\to M$ such that $f_1$ and $f_2$ are extremal epimorphisms and $f$, as a frame homomorphism, is the pushout of $f_1\oplus f_2$ along $q_L$ in $\mathsf{Frm}$.
\end{proposition}

 \begin{proof}
 Let $f\colon L\to M$ be such that $f_1$ and $f_2$ are onto and $f$ is the pushout of $f_1\oplus f_2$ along $q_L$ in $\mathsf{Frm}$. In order to check  that $f$ is an epimorphism, let $g,h\colon M\to N$ be biframe homomorphisms such $g\cdot f=h\cdot f$. Then, for $i=1,2$, as $g_i\cdot f_i=h_i\cdot f_i$  and $f_i$ is epic, we have that $g_i=h_i$. As $e_{M_1}$ and $e_{M_2}$ are jointly epic and, for $i=1,2$
 \[
 g\cdot e_{M_i}=e_{N_i}\cdot g_i=e_{N_i}\cdot h_i=h\cdot e_{M_i},
\]
we conclude that $g=h$.

For extremality, let $f=m\cdot h$ with $m\colon N\to M$ monic. As $f_i$ is epic, so is $m_i$. Thus, $m_i$ is a bijective frame homomorphism, that is, an isomorphism. Then we have, for $i=1,2$,
\[
\begin{aligned}
q_N\cdot m_1^{-1}\oplus m_2^{-1}\cdot f_1\oplus f_2\cdot\iota_{L_i}&=q_N\cdot (m_1^{-1}\cdot f_1)\oplus (m_1^{-1}\oplus f_2)\cdot \iota_{L_i}\\
&=q_N\cdot h_1\oplus h_2\cdot \iota_{L_i}\\
&=q_N\cdot\iota_{N_i}\cdot h_i\\
&=e_{N_i}\cdot h_i\\
&=h\cdot e_{L_i}\\
&=h\cdot q_L\cdot\iota_{L_i}.
\end{aligned}
\]
In consequence,
\[
q_N\cdot m_1^{-1}\oplus m_2^{-1}\cdot f_1\oplus f_2=h\cdot q_L,
\]
as coproduct injections are jointly  epic.
Furthermore, as, for $i=1,2$,
\[
\begin{aligned}
m\cdot q_N\cdot m_1^{-1}\oplus m_2^{-1}\cdot\iota_{M_i}&=m\cdot q_N\cdot\iota_{N_i}\cdot m_{i}^{-1}\\
&=m \cdot e_{N_i}\cdot m_i^{-1}\\
&=e_{M_i}\cdot m_i\cdot m_i^{-1}\\
&=e_{M_i}=q_M\cdot \iota_{M_i},
\end{aligned}
\]
it is the case that $m\cdot q_N\cdot m_1^{-1}\oplus m_2^{-1}=q_M$. Therefore
\[
\xymatrix{
L_1\oplus L_2\ar[r]^-{q_L}\ar[d]_{f_1\oplus f_2} & L_0\ar[d]^{h}\ar@/^1.5pc/[ddr]^{f}&\\
f[L_1]\oplus f[L_2]\ar@/_1.5pc/[drr]_{q_M}\ar[r]_-{q_N\cdot m_1^{-1}\oplus m_2^{-1}}&N_0\ar@{.>}[dr]^{m}&\\
&&M_0}
\]
commutes. In conclusion, $m$ is an isomorphism, as the inner square satisfies the universal property of the pushout of $q_L$ and $f_1\oplus f_2$.

 On the other hand, if $f$ does not satisfy any of the conditions, the decomposition $f=e\cdot \overline{f}$ from Proposition \ref{factor}, with $e$ a non-isomorphic mono, shows that $f$ is not an extremal epimorphism.
\end{proof}
 
 \section{The lattice of subbilocales of a biframe}\label{sectionlatq}

There is a natural preorder for extremal epimorphisms: for any $f\colon L\to S$ and $g\colon L\to T$, $f\lesssim g$ if there exists a homomorphism $h\colon T\to S$ such that $h\cdot g=f$. Of course, this is equally valid for the category $\mathsf{BiFrm}$ as well as for $\mathsf{Frm}$ and, in fact, for any category. For any $L$, whether a frame of a biframe, we will denote by $\mathcal{S}(L)$ posetal reflection of this preorder, i.e., the partially ordered set of equivalence classes $\langle f \rangle$ of extremal epimorphisms $f\colon L\to M$. The meet of $(f_i\colon L\to M_i)_{i\in I}$ is given by the canonical morphism $f\colon L\to M$ into the colimit of the diagram. In the case of frames, the elements of $\mathcal{S}(L)$ are called \emph{sublocales}, as they are the (extremal) subobject in the dual category $\mathsf{Loc}$ of locales. Sublocales are represented in frame theory in several different ways: sublocale sets, nuclei, frame quotients, \dots. In this paper, we will make use of the following: $\mathcal{S}(L)^{\rm op}$ is isomorphic to the frame $\mathcal{C}L$ of frame congruences of $L$. Going back to bispaces and mimicking the terminology from frame  theory, for a  biframe $L$, we are going to call the elements of $\mathcal{S}(L)$, that is, its extremal subobjects in the dual category of $\mathsf{BiFrm}$, \emph{subbilocales}. 

\begin{remark}
Admittedly, this term, subbilocale, has been introduced in the literature a few times with different meanings. Given a biframe $L$, Picado and Pultr in \cite{PP14} and Nxumalo in \cite{N23}  use this term to refer to a biframe $(S_0,S_1,S_2)$ where $S_0$ is a sublocale set of $L_0$ with an associated frame quotient $\nu_S\colon L_0\to S_0$ and $S_i=\nu_S[L_i]$ for $i=1,2$. This is weaker than our notion. Ferreira, Guti\'errez Garc\'\i a and Picado in \cite{FGP11} use it in an even weaker sense, as they refer with it to \emph{regular} subobjects in the dual category, which does not even require that $S_i=\nu_S[L_i]$ for $i=1,2$.
\end{remark}

If a biframe map $f\colon L\to M$ is an extremal epimorphism then $f\colon L_0\to M_0$ is an extremal epimorphism in $\mathsf{Frm}$, but the converse does not hold in general. For instance, consider a biframe $L$ where $L_0$ is a proper quotient of $L_1\oplus L_2$. Then the  biframe map $f\colon(L_1\oplus  L_2,L_1,L_2)\to L$ given by the frame quotient map $f\colon L_1\oplus L_2\to L_0$ is not an extremal epimorphism,  as it is a non-isomorphic monomorphism.

However, any extremal epimorphism $f\colon L_0 \to S$ in $\mathsf{Frm}$ gives rise to an extremal epimorphism in $\mathsf{BiFrm}$ in a slightly more refined way. Note that $f$ determines a biframe homomorphism $f\colon L\to (S,f[L_1],f[L_2])$. By the factorization in Proposition \ref{factor}, we have the commutative diagram
\[
\xymatrix{
L\ar[rr]^-{f}\ar[dr]_-{\overline{f}}&&(S,f[L_1],f[L_2])\\
&(f(L),f[L_1],f[L_2])\ar[ur]_e&
}
\]
where $\overline{f}$ an extremal epimorphism.

\begin{proposition}
For any biframe $L$, the assignment $\langle f\rangle\mapsto \langle\overline{f}\rangle$ defines a closure operator $B_L$ on $\mathcal{S}(L_0)$.
\end{proposition}

\begin{proof}
Let $f\colon L_0\to S$ and $g\colon L_0 \to T$ be a frame quotient maps such that $f\lesssim  g$, that is, there exists a frame homomorphism $h\colon T\to S$ such that $h\cdot g=f$. We have the following two pushouts diagrams:
\[
\xymatrix{
L_1\oplus L_2\ar[r]^-{q_L}\ar[d]_{f_1\oplus f_2} & L_0\ar[d]^{\overline{f}}\\
f[L_1]\oplus f[L_2]\ar[r]_-{\overline{q_L}}&f(S)}
\qquad
\xymatrix{
L_1\oplus L_2\ar[r]^-{q_L}\ar[d]_{g_1\oplus g_2} & L_0\ar[d]^{\overline{g}}\\
g[L_1]\oplus g[L_2]\ar[r]_-{\overline{q_L}^\prime}&g(T)}
\]
Since
\[
h_1\oplus h_2\cdot g_1\oplus g_2=f_1\oplus f_2,
\]
 where $h_i$ is the restriction of $h$ to $g[L_i]\to f[L_i]$, 
there exists $\widetilde{h}\colon g(T)\to f(S)$ such that 
\[
\xymatrix@R+1.pc@C+1.pc{
L_1\oplus L_2\ar@/_4.5pc/[dd]_{f_1\oplus f_2}\ar[r]^-{q_L}\ar[d]_{g_1\oplus g_2} & L_0\ar[d]^{\overline{g}}\ar@/^1.5pc/[ddr]^{\overline{f}}&\\
g[L_1]\oplus g[L_2]\ar[r]_-{\overline{q_L}^\prime}\ar[d]_{h_1\oplus h_2}&g(T)\ar@{.>}[dr]^{\widetilde{h}}&\\
f[L_1]\oplus f[L_2]\ar[rr]_{\overline{q_L}}&&f(S)}
\]
commutes. Thus, $\widetilde{h}\cdot \overline{g}=\overline{f}$, that is, $\overline{f}\lesssim\overline{g}$. Therefore $B_L$ is monotone.

Further, as
\[
\xymatrix{
L_0\ar[rr]^{f}\ar[dr]_{\overline{f}}&&S\\
&f(L)\ar[ur]_e&
}
\]
commutes, we have $f\lesssim\overline{f}$. Thus $B_L$ is inflationary. Finally, for idempotence, since $\overline{f}$ is an extremal epimorphism and $e^\prime$ a monomorphism in the commuting diagram 
\[
\xymatrix{
L\ar[rr]^-{\overline{f}}\ar[dr]_-{\overline{\overline{f}}}&&(f(L),f[L_1],f[L_2])\\
&(\overline{f}(f(L)),f[L_1],f[L_2])\ar[ur]_{e^\prime}&
}
\]
then $e^\prime$ is an isomorphism. In consequence $\overline{f}\simeq\overline{\overline{f}}$.
\end{proof}

\begin{corollary}\label{completequotients}
For any biframe $L$, $\mathcal{S}(L)$ is a complete lattice.
\end{corollary}

\begin{proof}
It is clear that $\mathcal{S}(L)$ is isomorphic to the poset of fixed elements of $\mathcal{S}(L_0)$ with respect to $B_L$.  Consequently, $\mathcal{S}(L)$ is a complete lattice.
\end{proof}

\section{subbilocales induced by components' quotients}\label{sectioninducedsubs}

In the previous section we have seen that a quotient of the ambient frame induces an extremal epimorphism in a biframe. In this section we show how a pair of quotients of the  component frames do it. This  approach will be useful in order to describe the lattice  of  subbilocales of a biframe.

There is a functor $\mathcal{S}$ from the category $\mathsf{Frm}$ of frames  to the category $\mathsf{coFrm}$ of coframes (with coframe homomorphism) that maps $L\mapsto\mathcal{S}(L)$ (see \cite{PJ82}). Specifically, for a frame homomorphism $k\colon L\to M$, $\mathcal{S}(k)\colon\mathcal{S}(L)\to\mathcal{S}(M)$ maps $\langle f\rangle$ to the equivalence class $\langle g\rangle$ of the pushout $g$ of $f$ along $k$. 
\begin{proposition}
Let $f\colon L\to T$ be an extremal epimorphism in $\mathsf{BiFrm}$. Then
\[
\langle f\rangle=\mathcal{S}(e_{L_1})\langle f_1\rangle\wedge\mathcal{S}(e_{L_2})\langle f_2\rangle.
\]
\end{proposition}

\begin{proof}
For $i=1,2$, let $f_i^\prime \colon L_0\to T_i^\prime$ be the pushout of $f_i$ along $e_{L_i}$.  Thus $\mathcal{S}(e_{L_i})\langle f_1\rangle=\langle f_i^\prime\rangle$. Also let $f^\prime$ be the canonical map $L_0\to T_0^\prime$ into the colimit of the diagram $(f_i^\prime\colon L_0\to T_i^\prime)_{i=1,2}$. Thus $\langle f_1^\prime\rangle\wedge\langle f_2^\prime\rangle=\langle f^\prime\rangle$.  Therefore, we need to show that $\langle f\rangle=\langle f^\prime\rangle$.

Since $f\cdot e_{L_i}=e_{T_i}\cdot f_i$, there exist  $r_i\colon T_i^\prime\to T_0$ such that 
\[
\xymatrix@R+1.pc@C+1.pc{
L_i\ar[r]^{e_{L_i}}\ar[d]_{f_i}	&L_0\ar[d]^{f_i^\prime}\ar@/^1.5pc/[ddr]^{f}&\\
T_i\ar[r]_{e_{L_i}^\prime}\ar@/_1.5pc/[drr]_{e_{T_i}}&T_i^\prime\ar[dr]^{r_i}&\\
&&T_0
}
\]
commutes, where the inner square is a pushout diagram. Thus $f=r_i\cdot f_i^\prime$. 
Therefore, there exists $k\colon T_0^\prime\to T_0$ such that 
\[
\xymatrix@R+1.pc@C+1.pc{
L_0\ar[d]_{f_1^\prime}\ar[r]^{f_2^\prime} \ar@/^1pc/[dr]^{f^\prime}& T_2^\prime\ar[d]^{p_2}\ar@/^1.5pc/[ddr]^{r_2}&\\
T_1^\prime \ar[r]_{p_1}\ar@/_1.5pc/[drr]_{r_1}&T_0^\prime\ar[dr]^k&\\
&&T_0}
\]
commutes, where, again, the inner square is a pushout diagram, that is, $p_1$ is the pushout of $f_2^\prime$ along $f_1^\prime$ and $p_2$ the pushout of $f_1^\prime$ along $f_2^\prime$. Therefore $\langle f^\prime\rangle\geq \langle f\rangle$, as $f=k\cdot f^\prime$. 
On the other hand, 
 let $q_T^\prime\colon T_1\oplus T_2\to T_0^\prime$ be the homomorphism that makes
\[
\xymatrix{
T_1\ar[r]^{\iota_{T_1}}\ar[d]_{e_{L_1}^\prime}&T_1\oplus T_2\ar[d]_{q_S^\prime}&T_2\ar[l]_{\iota_{T_2}}\ar[d]_{e_{L_2}}\\
T_1^\prime\ar[r]_{p_1}&T_0^\prime&T_2^\prime\ar[l]^{p_2}
}
\]
commute. As
\[
\begin{aligned}
q_T^\prime\cdot f_1\oplus f_2\cdot \iota_{L_i}&=q_T^\prime\cdot\iota_{T_i}\cdot f_i\\
&=p_i\cdot e_{L_i}^\prime\cdot f_i\\
&=p_i\cdot f_i^\prime\cdot e_{L_i}\\
&=f^\prime \cdot e_{L_i}\\
&=f^\prime \cdot q_L\cdot \iota_{L_i}
\end{aligned}
\]
for $i=1,2$, we have that $q_S^\prime\cdot f_1\oplus f_2=f^\prime \cdot q_L$. By Proposition \ref{extremalchar} there exists $l\colon T_0\to T_0^\prime$ making
\[
\xymatrix{
L_1\oplus L_2\ar[r]^{q_L}\ar[d]_{f_1\oplus f_2}&L_0\ar[d]^{f}\ar@/^1.5pc/[ddr]^{f^\prime}&\\
T_1\oplus T_2\ar[r]_{q_T}\ar@/_1.5pc/[drr]_{q_T^\prime}&T_0\ar[dr]_l&\\
&&S_0^\prime}
\]
commute, where the inner square is a pushout diagram. As $f^\prime =l\cdot f$, we conclude that $\langle f\rangle=\langle f^\prime\rangle$.
\end{proof}

\begin{proposition}
Let $L$ be a biframe. For any $\langle f_1\rangle\in\mathcal{S}(L_1)$ and $\langle f_2\rangle\in\mathcal{S}(L_2)$ there exists an extremal epimorphism $f\colon L\to M$ in $\mathsf{BiFrm}$ such that $\langle f_0\rangle=\mathcal{S}(e_{L_1})\langle f_1\rangle\wedge\mathcal{S}(e_{L_2})\langle f_2\rangle$.
\end{proposition}

\begin{proof}
For $i=1,2$, let $f_i\colon L_i\to S_i$ be extremal epimorphisms. Let $f_i^\prime\colon L_0\to S_i^\prime$ be the pushout of $f_i$ along $e_{L_i}$ and $g\colon L_0\to T_0$ be the canonical map into the pushout of the diagram $(f_i^\prime\colon L_0\to S_i^\prime)_{i=1,2}$. Since both $f_1^\prime$ and $f_2^\prime$ are onto, so is $g$. Consequently the union of $T_1=g[e_{L_1}[L_1]]$ and $T_2=g[e_{L_2}[L_2]]$ forms a subbasis of $T_0$. Thus $T=(T_0,T_1,T_2)$ is a biframe. Further, if we denote by $g_i$  the restriction of $g$ to $g_i\colon L_i\to T_i $, that is, $e_{T_i}\cdot g_i$ is the canonical extremal epi-mono factorization of $g\cdot e_{L_i}$ in $\mathsf{Frm}$, we have that $g\cdot e_{L_i}= e_{T_i}\cdot g_i$. Accordingly $g\colon L\to T$ is a biframe homomorphism.

We will show now that $g\colon L\to T$ is an extremal epimorphism in $\mathsf{BiFrm}$. Since $g_1$ and $g_2$ are extremal epimorphisms, by Proposition \ref{extremalchar}, we only have to check that $g$ is the pushout of $g_1\oplus g_2$ along $q_L$.

First, for $i,j=1,2$ such that $i\neq j$, let $e_{L_i}^\prime$ is the pushout of $e_{L_i}$ along $f_i$ and let $p_i$ be the push out of $f_j^\prime$ along $f_i^\prime$. More graphically, these two diagrams
\[
\xymatrix{
L_i\ar[r]^{e_{L_i}}\ar[d]_{f_i} & L_0\ar[d]^{f_i^\prime}\\
S_i\ar[r]_{e^\prime_{L_i}}&S_i^\prime
}
\quad
\xymatrix{
L_0\ar[r]^{f^\prime_1}\ar[d]_{f^\prime_2} & S_1^\prime\ar[d]^{p_1}\\
S_2^\prime\ar[r]_{p_2}&T_0
}
\]
are pushouts.

Since
\[
\begin{aligned}
p_i\cdot e_{L_i}^\prime\cdot f_i&=p_i\cdot f_i^\prime e_{L_i}\\
&=g\cdot e_{L_i}\\
&=e_{T_i}\cdot g_i
\end{aligned}
\]
 and extremal epimorphisms and monomorphisms are orthogonal in $\mathsf{Frm}$, there exists $h_i\colon S_i\to T_i$ making 
\[
\xymatrix{
L_i\ar[r]^{f_i}\ar[dd]_{g_i}&S_i\ar[d]^{e_{L_i}^\prime}\ar[ddl]_-{h_i}\\
&S_i^\prime\ar[d]^{p_i}\\
T_i\ar[r]_{e_{T_i}}&T_0
}
\]
commute.

Now, in order to show that $g$ is indeed the pushout of $g_1\oplus g_2$ along $q_L$, let $n\colon L_0\to R$ and $m\colon T_1\oplus T_2\to R$ be frame homomorphisms such that $n\cdot q_L=m\cdot g_1\oplus g_2$. Then 
\[
\begin{aligned}
n\cdot e_{L_i}&=n\cdot q_L\cdot \iota_{L_i}\\
&=m\cdot g_1\oplus g_2\cdot \iota_{L_i}\\
&=m\cdot \iota_{T_i}\cdot g_i\\
&=m\cdot \iota_{T_i}\cdot h_i\cdot f_i.
\end{aligned}
\]
Thus there exists $k_i\colon S_i^\prime\to R$ such that 
\[
\xymatrix{
L_i\ar[r]^{e_{L_i}}\ar[d]_{f_i}&L_0\ar[d]^{f_i^\prime}\ar@/^1.5pc/[ddr]^{n}&\\
S_i\ar[r]_{e_{L_i}^\prime}\ar@/_1.5pc/[drr]_{m\cdot \iota_{T_i}\cdot h_i}&S_i^\prime\ar[dr]_{k_i}&\\
&&R}
\]
commutes. As $k_1\cdot f_1^\prime=n=k_2\cdot f_2^\prime$, there exists a frame homomorphism $l\colon T_0\to R$ making
\[
\xymatrix@R+1.pc@C+1.pc{
L_0\ar[r]^{f_2^\prime}\ar[d]_{f_1^\prime}\ar[dr]^{g}&S_2^\prime\ar[d]^{p_2}\ar@/^1.5pc/[ddr]^{k_2}&\\
S_1^\prime\ar[r]_{p_1}\ar@/_1.5pc/[drr]_{k_1}&T_0\ar[dr]^l&\\
&&R}
\]
commute. Thus we have that $l\cdot g=k_i\cdot f_i^\prime =n$. On the other hand, we have that
\[
\begin{aligned}
m\cdot \iota_{T_i}\cdot h_i&=k_i\cdot e_{L_i}^\prime\\
&=l\cdot p_i\cdot e_{L_i}^\prime\\
&=l\cdot e_{T_i}\cdot h_i\\
&=l\cdot q_T\cdot \iota_{T_i}\cdot h_i
\end{aligned}
\]
for $i=1,2$. As $h_i$ is epic, $m\cdot\iota_{T_i}=l\cdot q_T\cdot \iota_{T_i}$. By the universal property of coproducts, we conclude that $m=l \cdot q_T$. In summary, we have that
\[
\xymatrix{
L_1\oplus L_2\ar[r]^{q_L}\ar[d]_{g_1\oplus g_2} & L_0\ar[d]^g\ar@/^1.5pc/[ddr]^n&\\
T_1\oplus T_2\ar[r]_{\overline{q_L}}\ar@/_1.5pc/[rrd]_m& T_0\ar[dr]^l&\\
&&R
}
\]
commutes. Therefore the universal property of pushouts holds for the inner square.
\end{proof}

The following corollary follows directly from the last two propositions.
\begin{corollary}\label{meetgenerated}
Let $L$ be a biframe. Then $\mathcal{S}(L)$ is isomorphic to the sub-$\wedge$-semilattice of $\mathcal{S}(L_0)$ generated by
\[
\mathcal{S}(e_{L_1})[\mathcal{S}(L_1)]\cup\mathcal{S}(e_{L_2})[\mathcal{S}(L_2)].
\]
\end{corollary}

Note this last corollary provides an alternative proof for Corollary \ref{completequotients}, as $\mathcal{S}(e_{L_i})[\mathcal{S}(L_i)]$ are closed under arbitrary meets.

\begin{example}
Let $\mathbf{3}=\{0<a<1\}$ be a three element frame. $\boldsymbol{3}$ has precisely four non-isomorphic extremal epimorphism forming the lattice depicted in Figure \ref{latsubs3}, where $\mathfrak{c}(a)$ and $\mathfrak{o}(a)$ are the closed and open quotients determined by $a$, respectively.
\begin{figure}\caption{The lattice of sublocales of $\boldsymbol{3}$}
\label{latsubs3}
\[
\xymatrix{
&1_{\mathcal{S}(\boldsymbol{3})}\ar@{-}[dl]\ar@{-}[dr]&\\
\mathfrak{o}(a)\ar@{-}[dr]&&\mathfrak{c}(a)\ar@{-}[dl]\\
&0_{\mathcal{S}(\boldsymbol{3})}&
}
\]
\end{figure}

Now consider the biframe $\boldsymbol{3}.\boldsymbol{3}$ where the ambient frame is $\boldsymbol{3}\oplus \boldsymbol{3}$ (depicted in Figure \ref{3plus3})
\begin{figure}\caption{The coproduct frame $\boldsymbol{3}\oplus \boldsymbol{3}$}
\label{3plus3}
\[
\xymatrix{
&1\oplus 1\ar@{-}[d]&\\
&a\oplus 1\vee 1\oplus a\ar@{-}[dl]\ar@{-}[dr]&\\
a\oplus 1\ar@{-}[dr]&&1\oplus a\ar@{-}[dl]\\
&a\oplus a\ar@{-}[d]&\\
&0\oplus 0&
}
\]
\end{figure}
 and the component frames are isomorphic copies of $\boldsymbol{3}$ determined by the coproduct injections. Then, by Corollary \ref{meetgenerated}, we have that $\mathcal{S}(\boldsymbol{3}.\boldsymbol{3})$ is the sub-$\wedge$-semilattice of $\mathcal{S}(\boldsymbol{3}\oplus\boldsymbol{3})$ generated by \[
 \{\mathfrak{o}(a\oplus 1),\mathfrak{o}(1\oplus a), \mathfrak{c}(a\oplus 1),\mathfrak{c}(1\oplus a), 1 , 0\}.
 \]
 
 The functor $\mathcal{S}\colon\mathsf{Frm}\to \mathsf{coFrm}$ behaves as one may expect with respect to  closed and open sublocales: for any frame homomorphism $f\colon L\to M$ and any element $a\in L$, we have that $\mathcal{S}(f)(\mathfrak{c}(a))=\mathfrak{c}(f(a))$ and $\mathcal{S}(f)(\mathfrak{o}(a))=\mathfrak{o}(f(a))$. Using this fact, 
by easy calculations, one can check that the diagram in Figure \ref{Q(3.3)} shows the lattice of subbilocales of $\boldsymbol{3}.\boldsymbol{3}$.
\begin{figure}\caption{The lattice of subbilocales of $\boldsymbol{3}.\boldsymbol{3}$ rotated 90 degrees clockwise.}
 \label{Q(3.3)}
\[
\xymatrix@d{&&\boldsymbol{3}.\boldsymbol{3}\ar@{-}[dll]\ar@{-}[dl]\ar@{-}[dr]\ar@{-}[drr]&&\\
\mathfrak{c}(a\oplus 1)\ar@{-}[d]\ar@{-}[dr]&\mathfrak{c}(1\oplus a)\ar@{-}[dl]\ar@{-}[drr]&&\mathfrak{o}(1\oplus a)\ar@{-}[dll]\ar@{-}[dr]&\mathfrak{o}(a\oplus 1)\ar@{-}[dl]\ar@{-}[d]\\
\mathfrak{c}(a\oplus 1\vee 1\oplus a)\ar@{-}[drr]&\mathfrak{c}(a\oplus 1)\wedge\mathfrak{o}(1\oplus a)\ar@{-}[dr]&&\mathfrak{c}(1\oplus a)\wedge\mathfrak{o}(a\oplus 1)\ar@{-}[dl]&\mathfrak{o}(a\oplus a)\ar@{-}[dll]\\
&&\boldsymbol{1}.\boldsymbol{1}&&
}
\]

\end{figure} 
This example shows that the lattice of subbilocales of a biframe may not be distributive, like in this case where $\mathcal{S}(\boldsymbol{3}.\boldsymbol{3})$ contains isomorphic copies of $N_5$.

In \cite{S92}, the congruence biframe of a biframe $L$ is defined to be the biframe $\mathcal{C}L=(\mathcal{C}_0L_0,\mathcal{C}_1L_1,\mathcal{C}_2L_2)$ such that, for $i=1,2$, $(\mathcal{C}L)_i=\mathcal{C}_iL_i$ is the subframe of $\mathcal{C}L_0$ (the frame of congruences of $L_0$) generated, as a frame, by closed congruences of $L_i$ and open congruences of $L_k$ ($i\neq k$) and $\mathcal{C}_0L_0$ is generated, as a frame, by $\mathcal{C}_1L_1\cup\mathcal{C}_2L_2$. Then, for a biframe $L$, given the dual order-isomorphism between the frame $\mathcal{C}L_0$ of congruences of the ambient frame and its coframe $\mathcal{S}(L_0)$ of sublocales, one may wonder whether the ambient frame $\mathcal{C}_0L_0$ of the congruence biframe $\mathcal{C}L$ is dually isomorphic to its lattice $\mathcal{S}(L)$ of subbilocales. The example above shows that this is not the case  in general, since $\mathcal{S}(\boldsymbol{3}.\boldsymbol{3})$ is not a coframe. For instance, $\mathfrak{o}(1\oplus a\vee a\oplus 1)$ appears represented in the ambient frame of the congruence biframe of $\boldsymbol{3}.\boldsymbol{3}$ but it has no counterpart in $\mathcal{S}(\boldsymbol{3}.\boldsymbol{3})$.

\end{example}

\section{Functoriality of $\mathcal{S}$}\label{sectionfunctorQ}
We focus now on showing that the assignment $L\mapsto\mathcal{S}(L)$ applied to biframes is functorial. 
\begin{proposition}
Let $f\colon L\to M$ be a biframe homomorphism and $q\colon L\to S$ be an extremal epimorphism. Let $q^\prime\colon M_0\to T_0$ be the pushout, in $\mathsf{Frm}$, of $q$ along $f$. Then 
$q^\prime\colon M\to T$,
where $T_i=q^\prime[M_i]$ and $T=(T_0, T_1,T_2)$, is a extremal epimorphism in $\mathsf{BiFrm}$
\end{proposition}

\begin{proof}First of all, note that since $q$ is onto so is $q^\prime$ and, consequently, as $M_1\cup M_2$ forms a subbasis of $M_0$, $q^\prime[M_1]\cup q^\prime[M_2]=T_1\cup T_2$ forms a subbasis of $q^\prime[M_0]=T_0$. Therefore, $T=(T_0,T_1,T_2)$ is actually a biframe. Further denote by $f^\prime$ the pushout of $f$ along $q$. Since $q_i$ is onto, $f^\prime[S_i]=f^\prime[q[L_i]]=q^\prime[f[L_i]]\subseteq T_i$ for $i=1,2$.  Thus $f^\prime$ determines a biframe map from $S$ to $T$.

Now we will show that $q^\prime$  is the pushout of $q_1^\prime\oplus q_2^\prime$ along $q_M$ For this purpose, let $n\colon T_1\oplus T_2\to R$ and $m\colon M_0\to R$ be frame homomorphisms such that $m\cdot q_M=n\cdot q_1^\prime\oplus q_2^\prime$.

Now, we have
\[
\begin{aligned}
m\cdot f\cdot q_L&=m\cdot q_M\cdot f_1\oplus f_2\\
&=n\cdot q_1^\prime\oplus q_2^\prime\cdot f_1\oplus f_2\\
&=n\cdot f_1^\prime\oplus f_2^\prime \cdot q_1\oplus q_2.
\end{aligned}
\]
Since $q$ is an extremal epimorphism, by Proposition \ref{extremalchar},
\[
\xymatrix{
L_1\oplus L_2\ar[r]^{q_L}\ar[d]_{q_1\oplus q_2}&L_0\ar[d]^q\\
S_1\oplus S_2\ar[r]_{q_S}&S_0}
\]
is a pushout diagram and therefore
 there exists a frame homomorphism $p\colon S_0\to R$ such that  $n\cdot f_1^\prime\oplus f_2^\prime=p\cdot q_S$ and $m\cdot f=p\cdot q$. By the latter equality and the fact that 
 \[
 \xymatrix{
 L_0\ar[d]_q\ar[r]^f& M_0\ar[d]^{q^\prime}\\
 S_0\ar[r]_{f^\prime} & T_0
 }
 \]
is a pushout, there exists $r\colon T_0\to R$ such that $m=r\cdot q^\prime$ and $p=r\cdot f^\prime$. Finally, we have that
\[
\begin{aligned}
n\cdot q_1^\prime\oplus q_2^\prime&=m\cdot q_M\\
&=r\cdot q^\prime \cdot q_M\\
&=r\cdot q_T\cdot q_1^\prime\oplus q_2^\prime
\end{aligned}
\]
and since $q_1^\prime\oplus q_2^\prime$ is epic we conclude that $r\cdot q_T=n$.
\end{proof}

Taking into account this last result, we can generalize the functor $\mathcal{S}\colon \mathsf{Frm}\to \mathsf{coFrm}$ to a functor from $\mathsf{BiFrm}$ to the category $\mathsf{CLat}$ of complete lattices and lattice homomorphisms.
\begin{theorem} For biframes $L$ and $M$ and biframe homomorphism $f\colon L\to M$,
the assignments $L\mapsto \mathcal{S}(L)$ and $f\mapsto \mathcal{S}(f)$ where $\mathcal{S}(f)(\langle q\rangle)=\langle q^\prime\rangle$ with $q^\prime$ the pushout of $q$ along $f$ for $q\in\mathcal{S}(L)$, defines a functor $\mathcal{S}\colon\mathsf{BiFrm}\to\mathsf{CLat}$.
\end{theorem}

\begin{proof}
It is easy to check that $\mathcal{S}$ preserves identity morphisms and composition.
\end{proof}

\section{The biframe version of Isbell's Density Theorem}\label{IsbellsDensity}

We finish this work by providing an analogue of Isbell's Density Theorem for biframes and we show that assigning to each biframe its smallest dense subbilocales is functorial if  we restrict the hom-sets to biframe homomorphism that satisfy a condition that generalizes the notion  of skeletal maps  to the bitopogical context.

\begin{definition}
We  will say that a biframe map $f\colon L\to M$ is \emph{dense} if the underlying frame homomorphism $f\colon L_0\to M_0$ is dense.
\end{definition}

\begin{lemma}\label{densepreservation}
Let $L$ be a biframe and $f\colon L_0\to S$ be an extremal epimorphism in $\mathsf{Frm}$. Then $f$  is dense iff and only if $\overline{f}$ is dense.
\end{lemma}

\begin{proof}
If $f$ is dense, then the density of $\overline{f}$ follows directly from the restriction of the extremal epi-mono factorization $f=e\cdot \overline{f}$ to the underlying ambient frames. If $\overline{f}$ is dense, the density of $f$ follows from the fact that $B_L$ is inflationary.
\end{proof}

The following is the main result in this section, a extension of the celebrated Isbell's Density Theorem to the context of biframes as pointfree bitopological spaces.

\begin{corollary}[Isbell's Density Theorem for biframes]
For any biframe $L$, $\langle\overline{\beta_{L_0}}\rangle$ is its least dense biframe subbilocale.
\end{corollary}

\begin{proof}
This follows from Lemma \ref{densepreservation} and the fact that $B_L$ is monotone.
\end{proof}

In order to have more concrete description  of the  least dense subbilocale of a biframe  $L$, consider the frame congruence $\mathcal{I}$ of $L_0$ generated by the set 
\[
\{(a, a^{\ast\ast})\}_{a\in L_1\cup L_2}\subseteq L_0\times L_0.
\]

\begin{proposition}
Let $L$ be a biframe. The frame quotient $p\colon L_0\to L_0/\mathcal{I}$ is the pushout of $\beta_{L_0}|_{L_1}\oplus \beta_{L_0}|_{L_2}$ along $q_L$.
\end{proposition}

\begin{proof}
First note that the map $q_i\colon \beta_{L_0}(L_i)\to L_0/\mathcal{I}$ be given  by $a^{\ast\ast}\mapsto \langle a^{\ast\ast}\rangle$ is a frame homomorphism. Therefore, there exists a frame homomorphism
\[
q\colon \beta_{L_0}[L_1]\oplus \beta_{L_0}[L_2]\to L_0/\mathcal{I}
\]
such that $q_i=q\cdot \iota_{\beta_{L_0}[L_i]}$. Therefore $q(a^{\ast\ast}\oplus 1)=\langle a^{\ast\ast}\rangle$ and $q(1\oplus b^{\ast\ast})=\langle b^{\ast\ast}\rangle$ for any $a\in L_1$ and $b\in L_2$. Consequently, we have 
\[
q\cdot(\beta_{L_0}|_{L_1}\oplus \beta_{L_0}|_{L_2})\cdot \iota_{L_i}(a)=\langle a^{\ast\ast}\rangle=p\cdot q_L\cdot \iota_{L_i}(a)
\]
for any $a\in L_i$. Thus $q\cdot(\beta_{L_0}|_{L_1}\oplus \beta_{L_0}|_{L_2})=p\cdot q_L$.

Now let $f\colon  L_0\to R$ and $g\colon \beta_{L_0}[L_1]\oplus \beta_{L_0}[L_2]\to R$ be frame homomorphisms such that
\[
f\cdot q_L=g\cdot (\beta_{L_0}|_{L_1}\oplus \beta_{L_0}|_{L_2}).
\]
 Then, for each $a\in L_i$, as clearly
\[
(\beta_{L_0}|_{L_1}\oplus \beta_{L_0}|_{L_2})\cdot\iota_{L_i}(a)=(\beta_{L_0}|_{L_1}\oplus \beta_{L_0}|_{L_2})\cdot\iota_{L_i}(a^{\ast\ast}),
\]
 we have
\[
\begin{aligned}
f(a)&=f(q_L\cdot\iota_{L_i}(a))\\
&=g\cdot (\beta_{L_0}|_{L_1}\oplus \beta_{L_0}|_{L_2})\cdot\iota_{L_i}(a)\\
&=g\cdot (\beta_{L_0}|_{L_1}\oplus \beta_{L_0}|_{L_2})\cdot\iota_{L_i}(a^{\ast\ast})\\
&=f(q_L\cdot\iota_{L_i}(a^{\ast\ast}))\\
&=f(a^{\ast\ast}).
\end{aligned}
\]
Consequently, there exists a frame homomorphism $h\colon L_0/\mathcal{I}\to R$ such that $h\cdot p=f$. Futher we have that 
\[
\begin{aligned}
g\cdot (\beta_{L_0}|_{L_1}\oplus \beta_{L_0}|_{L_2})&=f\cdot q_L\\
&=h\cdot p\cdot q_L\\
&=h\cdot q\cdot(\beta_{L_0}|_{L_1}\oplus \beta_{L_0}|_{L_2}).
\end{aligned}
\]
Finally, since $\beta_{L_0}|_{L_1}\oplus \beta_{L_0}|_{L_2}$ is epic, we conclude that $h\cdot q=g$.
\end{proof}

Given a biframe $L$, we can replace $L_0/\mathcal{I}$ in the construction above by $q_\ast[L_0/\mathcal{I}]\subseteq L_0$ where $q_\ast$ is the right adjoint of $q$. Simply note that, as $q$ is onto, its right adjoint $q_\ast$ is injective and consequently $L_0/\mathcal{I}$ is isomorphic to its image $q_\ast[L_0/\mathcal{I}]$.  Of course, this arises from the standard equivalence between \emph{sublocale homomorphisms}, that is, onto frame homomorphism, and \emph{sublocale sets}. Now we have a frame quotient
\[
\beta_L\colon L_0\to q_\ast[L_0/\mathcal{I}]
\]
where $\beta_L(a)=q_\ast(q(a))$ for all $a\in L_0$.
 Note that
 \[
 \beta_L(a)=q_\ast(q(a))=q_\ast(\langle a^{\ast\ast}\rangle)=a^{\ast\ast}
 \]
  for any $a\in L_1\cup L_2$. We can define a biframe
 \[
 \mathfrak{B} L=\left(q_\ast[L_0/\mathcal{I}], \beta_{L_0}[L_1], \beta_{L_0}[L_2]\right)
  \]
with the corresponding biframe map
\[
\beta_{L}\colon L\to \mathfrak{B} L
\]
as the least dense subbilocale  of $L$. Indeed, $\beta_L$ is the pushout of $\beta_L|_{L_1}\oplus \beta_L|_{L_2}$ along $q_L$.

Now we focus on how to restrict the homsets so that the construction above becomes functorial. For this purpose the following lemma will be useful.

\begin{lemma}\label{maprestriction}
Let $L$ and $M$ be biframes, $f\colon L_0\to S$ and $g\colon M_0\to R$ be frame extremal epimorphisms, and $h\colon L\to M$ be a biframe map. Let
\[
\overline{f}\colon L\to\widetilde{S}=(f(L),f[L_1],f[L_2])
\]
and
\[
 \overline{g}\colon M\to \widetilde{R}=(g(M),g[M_1],g[M_2])
\]
be the extremal epimorphisms induced by $f$ and $g$, respectively, as introduced in Section \ref{sectionlatq}. Then there exists a biframe map $\widehat{h}$ such that
\[
\xymatrix@R+1.pc@C+1.pc{
L_0\ar[r]^h\ar[d]_{\overline{f}} &M_0\ar[d]^{\overline{g}}\\
\widetilde{S}\ar[r]_{\widehat{h}}&\widetilde{R}
}
\]
commutes if and only if, for $i=1,2$, there exist $\widehat{h}_i$ such that the diagrams
\[
\xymatrix@R+1.pc@C+1.pc{
L_i\ar[r]^{h_i}\ar[d]_{f|_{L_i}} &M_i\ar[d]^{g|_{M_i}}\\
f[L_i]\ar[r]_{\widehat{h}_i}&g(M_i)
}
\]
commute.
\end{lemma}

\begin{proof} In order to check sufficiency, for $i=1,2$, let $\widehat{h}_i\colon f[L_i]\to g(M_i)$ be frame homomorphisms satisfying
\[
g|_{M_i}\cdot h_i=\widehat{h}_i\cdot f|_{L_i}.
\]
Then we have that
\[
\begin{aligned}
\overline{g}\cdot h\cdot q_L\cdot \iota_{L_i}& =\overline{g}\cdot q_M\cdot \iota_{M_i}\cdot h_i\\
&=\overline{q_M}\cdot (g|_{M_1}\oplus g|_{M_2})\cdot\iota_{M_i}\cdot h_i\\
&= \overline{q_M}\cdot \iota_{g(M_i)}\cdot g|_{M_i}\cdot h_i\\
&=\overline{q_M}\cdot \iota_{g(M_i)}\cdot \widehat{h}_i\cdot f|_{L_i},
\end{aligned}
\]
where $\overline{q_M}$ is the pushout of $q_M$ along $g|_{M_1}\oplus g|_{M_2}$.
It follows that
\[
\begin{aligned}
\overline{g}\cdot h\cdot q_L\cdot \iota_{L_i}&=\overline{q_M}\cdot (\widehat{h}_1\oplus\widehat{h}_2)\cdot \iota_{f[L_i]}\cdot f|_{L_i}\\
&=\overline{q_M}\cdot (\widehat{h}_1\oplus\widehat{h}_2)\cdot (f|_{L_1}\oplus f|_{L_2})\oplus\iota_{L_i}.
\end{aligned}
\]
As coproduct injections are jointly epic, we conclude that 
\[
\overline{g}\cdot h\cdot q_L=\overline{q_M}\cdot(\widehat{h}_1\oplus\widehat{h}_2)\cdot (f|_{L_1}\oplus f|_{L_2}).
\]

Consequently, there exists a unique frame map $\widehat{h}$ making 
\[
\xymatrix@R+1.pc@C+1.pc{
L_1\oplus L_2\ar[r]^{q_L}\ar[d]_{f|_{L_1}\oplus f_|{L_2}}&L_0\ar[d]^{\overline{f}}\ar@/^1.5pc/[ddr]^{\overline{g}\cdot h}&\\
f[L_1]\oplus f[L_2]\ar[r]_{\overline{q_L}}\ar[d]_{\widehat{h}_1\oplus\widehat{h}_2}&f(L)\ar@{.>}[dr]_{\widehat{h}}&\\
g(M_1)\oplus g(M_2)\ar[rr]_{\overline{q_M}}&&g(M)}
\]
commute, where the inner square is a pushout diagram. Therefore
\[
\begin{aligned}
\widehat{h}\cdot e_{f[L_i]}&=\widehat{h}\cdot \overline{q_{L}}\cdot \iota_{f[L_i]}\\
&=\overline{q_M}\cdot (\widehat{h}_1\oplus\widehat{h}_2)\cdot \iota_{f[L_i]}\\
&=\overline{q_M}\cdot\iota_{g(M_i)}\cdot\widehat{h}_i\\
&=e_{g(M_i)}\cdot\widehat{h}_i.
\end{aligned}
\]
We conclude that $\widehat{h}$ is biframe map $\widetilde{S}\to\widetilde{R}$. Finally $\overline{g}\cdot h=\widehat{h}\cdot \overline{f}$.

Necessity follows easily by restricting $\widehat{h}$ to the component frames.
\end{proof}

The following theorem is an extension to the context of biframes of a result by Banaschewski and Pultr in \cite{BP96}.

\begin{theorem}
Let $f\colon L\to M$ be a biframe map. There exists a biframe map $\hat{f}$ such that  
\[
\xymatrix@R+1.pc@C+1.pc{
L\ar[d]_{\beta_{L}}\ar[r]^{f} & M\ar[d]^{\beta_{M}}\\
\mathfrak{B}{L} \ar[r]_{\hat{f}}&\mathfrak{B}{M}
} 
\]
commutes if and only if $f(a^{\ast\ast})\leq f(a)^{\ast\ast}$ for all $a\in L_1\cup L_2$ (where the pseudocomplements are computed in the corresponding ambient frames).
\end{theorem}

\begin{proof} For sufficiency, let $\widehat{f}_i\colon \beta_{L_0}[L_i]\to\beta_{M_0}[M_i]$ be defined as 
\[
\widehat{f}_i(a)=f(a)^{\ast\ast}
\]
 for all $a\in \beta_{L_0}[L_i]$. We need to check that $\widehat{f}_i$ is a frame homomorphism. Note first that, since $L_i$ is a subframe of $L_0$, $\beta_{L_0}[L_i]$ is a subframe of $\mathfrak{B}L_0$. In order to check that $\widehat{f}_i$ preserves finite infima, note that
\[
\begin{aligned}
\widehat{f}_i(a\wedge b)&=f(a\wedge b)^{\ast\ast}\\
&=(f(a)\wedge f(b))^{\ast\ast}\\
&=f(a)^{\ast\ast}\wedge f(b)^{\ast\ast}\\
&=\widehat{f}_i(a)\wedge \widehat{f}_i(b)
\end{aligned}
\]
for any $a,b\in \beta_{L_0}[L_i]~.$ To check that $\widehat{f}_i$ preserves suprema, let $A\subseteq \beta_{L_0}[L_i]$. We have that
\[
\begin{aligned}
\widehat{f}_i\left(\tbigvee^{\beta_{L_0}[L_i]}A\right)&=\widehat{f}_i\left(\tbigvee^{\mathfrak{B}L_0} A\right)\\
&=f\left(\left(\tbigvee^{L_0}A\right)^{\ast\ast}\right)^{\ast\ast}.
\end{aligned}
\]
Thus, if $f(a^{\ast\ast})\leq f(a)^{\ast\ast}$ holds for all $a\in L_i$, we have
\[
\widehat{f}_i\left(\tbigvee^{\beta_{L_0}[L_i]}A\right)\leq f\left(\tbigvee^{L_0}A\right)^{\ast\ast\ast\ast}
= f\left(\tbigvee^{L_0}A\right)^{\ast\ast}.
\]
Furthermore
\[
\begin{aligned}
\tbigvee_{a\in A}^{\beta_{M_0}[M_i]}\widehat{f}_i(a)&=\tbigvee_{a\in A}^{\mathfrak{B}M_0}f(a)^{\ast\ast}\\
&=\left(\tbigvee^{M_0}_{a\in A}f(a)^{\ast\ast}\right)^{\ast\ast}\\
&=\left(\tbigvee^{M_0}_{a\in A}f(a)\right)^{\ast\ast}\\
&= f\left(\tbigvee^{L_0}A\right)^{\ast\ast}.
\end{aligned}
\]
We conclude that
\[
\widehat{f}_i\left(\tbigvee^{\beta_{L_0}[L_i]}A\right)\leq \tbigvee_{a\in A}^{\beta_{M_0}[M_i]}\widehat{f}_i(a)
\]
and as $\widehat{f}_i$ is clearly monotone, we have that
\[
\widehat{f}_i\left(\tbigvee^{\beta_{L_0}[L_i]}A\right)= \tbigvee_{a\in A}^{\beta_{M_0}[M_i]}\widehat{f}_i(a).
\]

Finally, as $f(a^{\ast\ast})\leq f(a)^{\ast\ast}$ for any $a\in L_i$, we have that
\[
f(a^{\ast\ast})\leq f(a^{\ast\ast})^{\ast\ast}\leq f(a)^{\ast\ast\ast\ast}=f(a^{\ast\ast}).
\] Therefore 
\[
\widehat{f}_i(\beta_{L}(a))=f(a^{\ast\ast})^{\ast\ast}= f(a)^{\ast\ast}=\beta_{M}(f(a)).
\]
We conclude that $\widehat{f}_i\cdot \beta_{L}=\beta_{M}\cdot f$. By Lemma \ref{maprestriction},there exists a biframe map $\hat{f}$ making
\[
\xymatrix@R+1.pc@C+1.pc{
L\ar[d]_{\beta_{L}}\ar[r]^{f} & M\ar[d]^{\beta_{M}}\\
\mathfrak{B}{L} \ar[r]_{\hat{f}}&\mathfrak{B}{M}
} 
\]
commute.

For necessity, if there exists $\widehat{f}$ such that $\widehat{f}\cdot \beta_{L}=\beta_{M}\cdot f$, for any $a\in L_i$, we have
\[
\widehat{f}\cdot \beta_{L}(a)=\widehat{f}(a^{\ast\ast})=f(a)^{\ast\ast}=\beta_{M}\cdot f(a).
\]
Hence also $\widehat{f}(a^{\ast\ast})=\widehat{f}(a^{\ast\ast\ast\ast})=f(a^{\ast\ast})^{\ast\ast}$. Thus finally $f(a^{\ast\ast})\leq f(a^{\ast\ast})^{\ast\ast}=f(a)^{\ast\ast}$.
\end{proof}

\end{document}